\documentclass[11pt,a4paper]{amsart}
\usepackage[margin=28mm]{geometry}
\usepackage{amsmath,amssymb,mathtools,amsthm}
\usepackage[T1]{fontenc}
\usepackage[expansion=false]{microtype}
\usepackage{hyperref}
\usepackage{needspace}

\theoremstyle{plain}
\newtheorem{theorem}{Theorem}[section]
\newtheorem{proposition}[theorem]{Proposition}
\newtheorem{lemma}[theorem]{Lemma}
\newtheorem{corollary}[theorem]{Corollary}

\newcommand{\dd}{\,d}
\newcommand{\ind}{\mathbf 1}
\newcommand{\common}{\mathcal M}
\newcommand{\exc}{\mathcal E}

\numberwithin{equation}{section}
\makeatletter
\@namedef{subjclassname@2020}{2020 Mathematics Subject Classification}
\makeatother

\hypersetup{
  pdftitle={Common invariant measures for beta-transformations: affine relations and squared bases},
  pdfauthor={Hang Zhao},
  pdfsubject={Common invariant measures for beta-transformations},
  colorlinks=false,
  pdfborder={0 0 0}
}
\title[Common invariant measures for beta-transformations]{Common invariant measures for beta-transformations:\\ affine relations and squared bases}
\author{Hang Zhao}
\address{School of Mathematics and Big Data, Neijiang Normal University, Neijiang 641100, Sichuan, China}
\email{hangzhaoscu@163.com}

\urladdr{https://orcid.org/0009-0003-2445-3477}
\thanks{\textbf{Acknowledgements.} This work was partially supported by the National Natural Science Foundation of China (No.\ 123B2006).}
\subjclass[2020]{Primary 37A05; Secondary 11K55, 37E05, 37A44}
\keywords{beta-transformation, common invariant measure, singular measure, irrational rotation, rational affine relation}
\date{}

\begin{document}
\begin{abstract}
We classify the common invariant probabilities of two beta-transformations in two families. In the first, the bases satisfy $\gamma=a\beta+c>1$ with $\beta>1$ irrational, $a,c\in\mathbb Q$ and $a+c\ne1$. In the second, they are an irrational
$\beta>1$ and its square. In both families every common singular invariant probability is the point mass at zero; combining this with the classification of coincident R\'enyi--Parry measures gives all common invariant probabilities, with no entropy, ergodicity or multiplicative-independence hypothesis.
\end{abstract}
\maketitle

\section{Introduction}

For $b>1$, the \textbf{beta-transformation} on $X\coloneqq[0,1)$ is
\begin{equation}\label{eq:map}
T_b(x)\coloneqq bx-\lfloor bx\rfloor.
\end{equation}
Let $\lambda$ be Lebesgue probability measure on $X$, let $\mathcal P(X)$ be the Borel probability measures on $X$, and put $\mathbb N\coloneqq\{1,2,\ldots\}$. For a Borel map $F$, \textbf{pushforward} is defined by $F_*\mu(E)\coloneqq\mu(F^{-1}E)$. Each $T_b$ has a unique absolutely continuous invariant probability $\nu_b$, the \textbf{R\'enyi--Parry measure}, equivalent to $\lambda$ \cite{Parry,Renyi}. The set of \textbf{common invariant probabilities} for two bases is
$$
\common(\beta,\gamma)\coloneqq
\{\mu\in\mathcal P(X):(T_\beta)_*\mu=(T_\gamma)_*\mu=\mu\}.
$$

The point mass at zero lies in this set for every pair of bases. Equality of the R\'enyi--Parry measures accounts for the
absolutely continuous part, but says nothing about common singular measures. Here we determine the singular part in two
families: bases related by a rational affine relation, and an irrational base paired with its square.
These two families overlap at quadratic irrational bases, but the squared-base result also covers algebraic bases of higher degree and transcendental bases.

\subsection{Main results}

Our first result concerns two bases linked by a rational affine relation.

\begin{theorem}\label{thm:singular}
Let $\beta>1$ be irrational and suppose
\begin{equation}\label{eq:affine-parameters}
\gamma=a\beta+c>1,\qquad a,c\in\mathbb Q,\qquad a+c\ne1.
\end{equation}
Every finite nonnegative Borel measure $\mu\perp\lambda$ invariant under $T_\beta$ and $T_\gamma$ is a nonnegative multiple of $\delta_0$.
\end{theorem}

The full classification uses the known quadratic family
\begin{equation}\label{eq:exceptional}
\exc\coloneqq\{b>1:b^2-ub-v=0\text{ for some }u,v\in\mathbb N,\ 1\le v\le u\}.
\end{equation}
The coincidence classification for R\'enyi--Parry measures is stated by Brown and Yin \cite[Theorem~1]{BY}: two
distinct noninteger bases have equal R\'enyi--Parry measures exactly in the adjacent-base family
\eqref{eq:exceptional}. The same classification is established by Huang and Wang \cite[Theorem~1.1]{HW}, whose
formulation we use below.

\Needspace{8\baselineskip}
Combining this classification with the preceding theorem gives all common invariant probabilities in the affine family.

\begin{corollary}\label{cor:affine-classification}
Under the assumptions of Theorem~\ref{thm:singular},
$$
\common(\beta,\gamma)=
\begin{cases}
\{(1-t)\delta_0+t\nu_\beta:0\le t\le1\},&|\beta-\gamma|=1\text{ and }\min\{\beta,\gamma\}\in\exc,\\[3pt]
\{\delta_0\}, &\text{otherwise}.
\end{cases}
$$
In the exceptional case, $\nu_\beta=\nu_\gamma$.
\end{corollary}

The bases in \eqref{eq:affine-parameters} are distinct, since equality would force $a=1,c=0$. The exception in the corollary can
occur only for $(a,c)=(1,1)$ or $(1,-1)$. In particular, an irrational base and a rational base have only $\delta_0$ in common.

The next theorem removes the affine hypothesis: it treats an irrational base and its square. Put $\varphi\coloneqq(1+\sqrt5)/2$.

\begin{theorem}\label{thm:square}
Let $\beta>1$ be irrational. Every finite nonnegative Borel measure singular with respect to $\lambda$
and invariant under both $T_\beta$ and $T_{ \beta^2}$ is a nonnegative multiple of $\delta_0$. Moreover,
\begin{equation}\label{eq:square-classification}
\common(\beta,\beta^2)=
\begin{cases}
\{(1-t)\delta_0+ t\nu_\varphi:0\le t\le1\}, &\beta=\varphi,\\[3pt]
\{\delta_0\},&\beta\ne\varphi.
\end{cases}
\end{equation}
Here $\nu_\varphi=\nu_{\varphi^2}$.
\end{theorem}

If $\beta$ is quadratic irrational then $\beta^2$ is a rational affine function of $\beta$, so Theorem~\ref{thm:singular}
already applies. The proof of Theorem~\ref{thm:square} instead uses the first-digit partition, and covers every irrational
$\beta$. The golden-ratio coincidence of absolutely continuous measures is known; see \cite[Corollary 1.3]{HW}. The
singular assertion, including the golden-ratio case, does not use that coincidence classification.

Theorem~\ref{thm:square} holds for every irrational base, and the exponent two is optimal in this uniformity:
for each integer $n\ge3$, Proposition~\ref{prop:higher-powers} constructs an irrational $\beta\in(2,3)$ such that
$T_\beta$ and $T_{\beta^n}$ have a nonzero common fixed point, so no statement with exponent $n\ge3$ can hold for
all irrational bases.

\subsection{Previous work and proof strategy}

Common invariant measures for integer multiplication maps form a classical problem going back to Furstenberg \cite{Furstenberg}. Rudolph \cite[Theorem 4.9]{Rudolph} proved rigidity for relatively prime integer bases, and Johnson \cite{Johnson} extended it to multiplicatively independent integer bases: under joint ergodicity and positive entropy for at least one map, the common measure is Lebesgue measure.

For noninteger bases, Bertrand-Mathis \cite{BM} studied certain Pisot pairs under Bernoulli or weak Bernoulli assumptions. Hochman and Shmerkin \cite[Corollary 1.11]{HS} proved that, for multiplicatively independent bases $b_1,b_2>1$ with $b_1$ Pisot, a common invariant probability is the common R\'enyi--Parry measure if all its $T_{b_2}$-ergodic components have positive entropy.

Our parameter assumptions differ, and our conclusions require no entropy hypothesis. In particular, the bases in Theorem~\ref{thm:square} are multiplicatively dependent.

Both proofs compare maps under which the original measure is invariant. For a rational affine relation, we compare the two orders of composition after an integer multiplication modulo one. The two orders differ by an integer power of a single irrational rotation. A signed telescoping sum is rotation-invariant, and its mass is a nonzero multiple of the first moment of the original measure. For squared bases, we compare $T_{\beta^2}$ with the second iterate of $T_\beta$. The corresponding rotation powers are indexed by nonnegative digits, so the telescoping measure is positive. In both cases singularity forces the rotation-invariant measure to vanish and then forces the original measure to be supported at zero.

Section~\ref{sec:prelim} records the measure-theoretic facts used in both arguments. Sections~\ref{sec:affine} and~\ref{sec:square}
prove the two classifications. Section~\ref{sec:scope} gives consequences and two questions.

\section{Preliminaries}\label{sec:prelim}
Throughout, measures are finite and nonnegative unless stated otherwise. We identify $X$ with $\mathbb R/\mathbb Z$ only as a
Borel space; for noninteger $b$, the map $T_b$ is not a group endomorphism of the circle. For $\theta\in\mathbb R$, define the \textbf{rotation}
$$
R_\theta:X\to X,\qquad R_\theta(x)\coloneqq\{x+\theta\}.
$$
We write $(R_\theta)_*$ for its pushforward on finite signed Borel measures and $I$ for the identity operator on this space.
The composition rule is $(R_\theta)_*(R_\phi)_*= (R_{\theta+\phi})_*$ for $\theta,\phi\in\mathbb R$.
For a Borel set $E$, the notation $\eta|_E$ denotes \textbf{restriction}.

We first record how the maps used below act on the two parts of the Lebesgue  decomposition.

\begin{lemma}\label{lem:types}
Suppose that $F:X\to X$ has finitely many affine interval branches of nonzero slope, with finitely many endpoints treated separately.
Then $F_*$ preserves absolute continuity and singularity relative to $\lambda$.  The two parts of the Lebesgue decomposition
of an $F$-invariant measure are separately $F$-invariant. Preservation of singularity also holds for finite signed measures.
\end{lemma}

\begin{proof}
The inverse image of a Borel null set is null on each branch, so absolute continuity is preserved. Each branch is the restriction of an affine homeomorphism of the real line and therefore maps Borel null sets to Borel null sets. The finite union of these images and the endpoint images is null, so singularity is preserved, also for the total variation of a signed measure. If $\mu=\mu_{\mathrm{ac}}+\mu_{\mathrm{s}}$ is
invariant, then $\mu=F_*\mu_{\mathrm{ac}}+F_*\mu_{\mathrm{s}}$ is another Lebesgue decomposition. Uniqueness gives invariance of both parts.
\end{proof}

We also use the following consequence of unique ergodicity for irrational rotations.

\begin{lemma}\label{lem:rotation}
If $\theta\notin\mathbb Q$ and a finite signed measure $\eta$ satisfies $(R_\theta)_*\eta= \eta$, then $\eta=\eta(X)\lambda$.
\end{lemma}

\begin{proof}
An irrational circle rotation is uniquely ergodic, with invariant probability $\lambda$; see \cite[Theorem 3.4]{McMullen}. This gives the statement for nonnegative measures by normalization. Since a rotation is invertible, uniqueness of the Jordan decomposition makes the positive and negative parts of an invariant signed measure separately invariant. Applying the nonnegative case to those parts proves the assertion.
\end{proof}

We use uniqueness of the absolutely continuous invariant probability $\nu_b$ for $T_b$ \cite[pp.~401--402]{Parry}. Uniqueness among
invariant probabilities equivalent to $\lambda$ suffices: if $\xi \ll\lambda$ is an invariant probability, then $(\xi+\nu_b)/2$ is invariant and equivalent to $\lambda$, so uniqueness gives $\xi=\nu_b$. We also use the following coincidence classification \cite[Theorem~1]{BY}, \cite[Theorem~1.1]{HW}:
\begin{equation}\label{eq:classical}
\begin{gathered}
1<b_1<b_2,\qquad b_1,b_2\notin\mathbb Z,\\
\nu_{b_1}=\nu_{b_2}\quad\Longleftrightarrow\quad b_2=b_1+1\ \text{ and }\ b_1\in\exc.
\end{gathered}
\end{equation}
This result is used only for the absolutely continuous parts of the two classifications.

\section{Rational affine relations}\label{sec:affine}

We begin with the signed measure that links the difference of the two compositions to the first moment.

\begin{proposition}\label{prop:rotation}
Let $\beta>1$ be irrational and assume \eqref{eq:affine-parameters}. Define
$$
q\coloneqq\min\{\ell\in\mathbb N:\ell a,\ell c\in\mathbb Z\},\qquad
p\coloneqq qa,\qquad r\coloneqq qc,\qquad D\coloneqq p+r-q.
$$
For every common invariant measure $\mu$, there is a finite signed measure $\sigma$ such that
\begin{equation}\label{eq:identity}
\sigma=DM\lambda,\qquad M\coloneqq\int_Xx\dd\mu(x).
\end{equation}
If $\mu\perp\lambda$, then $\sigma$ is singular as well.
\end{proposition}

\begin{proof}
Such an integer $q$ exists because $a,c\in\mathbb Q$. We have $q\gamma=p\beta+r$ and $D=q(a+c-1)\ne0$. Put
$$
T\coloneqq T_\beta,\quad S\coloneqq T_\gamma,\quad
d(x)\coloneqq\lfloor\beta x\rfloor,\quad e(x)\coloneqq\lfloor\gamma x \rfloor,\quad
j(x)\coloneqq qe(x)-pd(x).
$$
The finite set $J\coloneqq j(X)\subset\mathbb Z$ indexes the Borel partition $E_j\coloneqq\{x\in X:j(x)=j\}$
for $j\in J$. The floor convention includes every digit endpoint. Direct calculation gives
$$
q\bigl(S(Tx)-T(Sx)\bigr)
\equiv\beta\bigl(qe(x)-pd(x)\bigr)-rd(x)
\equiv\beta j(x)\pmod1.
$$
Define $\pi_q(x)\coloneqq\{qx\}$. Since $rd(x)$ is an integer,
\begin{equation}\label{eq:quotient}
\pi_q\circ S\circ T= R_{j\beta}\circ\pi_q\circ T \circ S
\quad\text{on }E_j.
\end{equation}
Put
$$
\eta_j\coloneqq(\pi_q\circ T\circ S)_*(\mu|_{E_j}),\quad j\in J.
$$
Common invariance implies invariance under both compositions, whether or not the maps commute. Therefore
\begin{equation}\label{eq:balance}
\sum_{j\in J}\eta_j=(\pi_q)_*\mu=\sum_{j\in J}(R_{j\beta})_*\eta_j.
\end{equation}

For each integer $j$, define an operator on finite signed measures by
$$
H_j\coloneqq
\begin{cases}
\displaystyle\sum_{\ell=0}^{j-1}(R_{\ell\beta})_*,&j>0,\\[3pt]
0,&j=0,\\[3pt]
\displaystyle-\sum_{\ell=j}^{-1}(R_{\ell \beta})_*,&j<0.
\end{cases}
$$
The composition rule for rotations gives, in every case,
$$
\bigl((R_\beta)_*-I\bigr)H_j=(R_{j\beta})_*-I,\qquad (H_j\eta)(X)=j\eta(X).
$$
It follows from \eqref{eq:balance} that
\begin{equation}\label{eq:sigma}
\sigma\coloneqq\sum_{j\in J}H_j\eta_j
\end{equation}
is $R_\beta$-invariant. It is finite, since $\|\sigma\|_{\mathrm{TV}}\le\sum_{j\in J}|j|\mu(E_j)<\infty$.
Lemma~\ref{lem:rotation} gives $\sigma= \sigma(X)\lambda$.

Invariance of $\mu$ gives
$$
\int_Xd\dd\mu=(\beta-1)M,\qquad
\int_Xe\dd\mu=(\gamma-1)M.
$$
The signed total mass is consequently
\begin{equation}\label{eq:moment}
\sigma(X)=\int_Xj\dd\mu
=\bigl(q(\gamma-1)-p(\beta-1)\bigr)M
=(p+r-q)M=DM.
\end{equation}
This proves \eqref{eq:identity}. If $\mu\perp \lambda$, each $\eta_j$ is singular by Lemma~\ref{lem:types}, since $\pi_q\circ T\circ S$ has finitely many affine branches with nonzero slope. Each summand in \eqref{eq:sigma} is a finite signed sum of rotations of singular measures. A finite union of Borel null sets carrying these rotated measures carries the total variation of $\sigma$, which proves the singularity assertion.
\end{proof}

\begin{proof}[Proof of Theorem~\ref{thm:singular}]
For a common singular measure, Proposition~\ref{prop:rotation} makes $\sigma=DM\lambda$ singular as well as absolutely continuous.
Hence $DM=0$. Since $D\ne0$, the nonnegative coordinate function has zero integral. Thus $\mu((0,1))=0$ and $\mu=\mu(X)\delta_0$.
\end{proof}

\begin{proof}[Proof of Corollary~\ref{cor:affine-classification}]
Let $\mu\in\common(\beta,\gamma)$. By Lemma~\ref{lem:types}, its Lebesgue decomposition has two
common invariant parts. Put $t\coloneqq\mu_{\mathrm{ac}}(X)$. Theorem~\ref{thm:singular} gives
$$
\mu=\mu_{\mathrm{ac}}+(1-t)\delta_0.
$$
If $t>0$, uniqueness of the absolutely continuous invariant probabilities gives
$$
\mu_{\mathrm{ac}}/t=\nu_\beta=\nu_\gamma.
$$
An integer $\gamma$ has $\nu_\gamma=\lambda$, whereas $\lambda$ is not invariant under
$T_\beta$. Indeed, for $k\coloneqq\lfloor\beta\rfloor$ and $\theta\coloneqq\beta-k\in(0,1)$,
\begin{equation}\label{eq:lebesgue-push}
\frac{d(T_\beta)_*\lambda}{d\lambda}
=\frac{k+\ind_{[0,\theta)}}{\beta},
\end{equation}
which is not constant. Thus both bases are nonintegers. Applying \eqref{eq:classical} in increasing order gives precisely the stated
exceptional case. Conversely, in that case the common  R\'enyi--Parry probability and $\delta_0$ give every stated mixture. Otherwise $t=0$.
\end{proof}

Corollary~\ref{cor:affine-classification} gives the following classification for rational gaps, with the known absolutely continuous exception.

\begin{corollary}\label{cor:gaps}
Let $\beta>1$ be irrational and let $c\in\mathbb Q_{>0}$. Then
$$
\common(\beta,\beta+c)=
\begin{cases}
\{(1-t)\delta_0+t\nu_\beta:0\le t\le1\},&c=1\text{ and }\beta\in\exc,\\[3pt]
\{\delta_0\},&\text{otherwise}.
\end{cases}
$$
In the exceptional case, the coefficient is uniquely given by $t=1-\mu(\{0\})$.
\end{corollary}

\begin{proof}
Apply Corollary~\ref{cor:affine-classification} with $a=1$. The smaller base is $\beta$, and the exceptional gap
must be one. Since $\nu_\beta$ is absolutely continuous, it has no atoms, which gives uniqueness of the coefficient.
\end{proof}

The common density in this exception is $(1+\beta^{-1} \ind_{[0,\alpha)})/(1+\alpha/\beta)$
with $\alpha\coloneqq\beta-u=v/\beta$, as computed in \cite[Section 2]{BY} and
\cite[Proposition 2.1]{HW}.

Two further consequences follow directly. For irrational $\beta>1$, integers $k\ge2$,
rational $c\ge0$ and rational $r>1$,
$$
\common(\beta,k\beta+c)=\{\delta_0\}
\quad\text{and}\quad
\common(\beta,r)=\{\delta_0\}.
$$
In the first case $a=k$ and $a+c>1$; in the second $a=0$ and $c=r>1$. Both gaps are
irrational, so neither pair has the exceptional gap of one.

\section{An irrational base and its square}\label{sec:square}

The second iterate $T_\beta^2$ differs from $T_{\beta^2}$ in general: on each first-digit branch the two maps differ by a power of a single rotation.

\begin{proof}[Proof of Theorem~\ref{thm:square}]
First let $\mu$ be a finite common singular measure. Define the digit bound and the branch measures by
$$
N\coloneqq\lfloor\beta\rfloor,\qquad
I_i\coloneqq\{x\in X:\lfloor\beta x\rfloor=i\},\qquad
\xi_i\coloneqq(T_{\beta^2})_*(\mu|_{I_i}),\quad 0\le i\le N.
$$
The half-open intervals $I_i$ partition $X$. On $I_i$ we have
\begin{equation}\label{eq:square-branch}
T_\beta^2x=\{\beta^2x-i\beta\}=R_{-i\beta}(T_{\beta^2}x).
\end{equation}
Invariance under $T_\beta$ implies invariance under $T_\beta^2$. Comparing the two invariant images in \eqref{eq:square-branch} gives
$$
\sum_{i=1}^{N}\bigl(I-(R_{-i\beta})_*\bigr)\xi_i=0.
$$
Define the finite nonnegative measure
\begin{equation}\label{eq:square-telescope}
\eta\coloneqq\sum_{i=1}^{N}\sum_{j=0}^{i-1} (R_{-j\beta})_*\xi_i.
\end{equation}
For $1\le i\le N$, the identity
$$
I-(R_{-i\beta})_* = \bigl(I-(R_{-\beta})_*\bigr)\sum_{j=0}^{i-1}(R_{-j\beta})_*
$$
gives $(R_{-\beta})_*\eta=\eta$. Lemma~\ref{lem:types} makes each $\xi_i$ singular, and
rotations preserve singularity, so $\eta\perp\lambda$. Since $\beta$ is irrational, Lemma~\ref{lem:rotation} gives
$\eta=\eta(X)\lambda$. Therefore $\eta=0$. Taking total masses and then using invariance under $T_\beta$, we obtain
$$
0=\eta(X)=\sum_{i=1}^{N}i\mu(I_i)
=\int_X\lfloor\beta x\rfloor\dd\mu(x)
=(\beta-1)\int_Xx\dd\mu(x).
$$
Nonnegativity and $\beta>1$ imply $\mu=\mu(X) \delta_0$.

For an arbitrary common invariant probability, Lemma~\ref{lem:types} and the singular assertion give $\mu=\mu_{\mathrm{ac}}+(1-t)\delta_0$, where
$t\coloneqq\mu_{\mathrm{ac}}(X)$. If $t>0$, then $\nu_\beta=\nu_{\beta^2}$. If $\beta^2$ is an integer, its R\'enyi--Parry measure is $\lambda$, contradicting
\eqref{eq:lebesgue-push}. Otherwise \eqref{eq:classical}, with $\beta<\beta^2$, gives $\beta^2=\beta+1$, so $\beta=\varphi$. Conversely, $\varphi\in\exc$
and $\varphi^2=\varphi+1$, so the known common measure $\nu_\varphi=\nu_{\varphi^2}$ gives exactly the mixtures in \eqref{eq:square-classification}.
\end{proof}

\Needspace{8\baselineskip}
The next proposition shows that no higher integer exponent gives the same singular rigidity for every irrational base.

\begin{proposition}\label{prop:higher-powers}
For every integer $n\ge3$, there are an irrational base $\beta_n$ and a point $x_n\in(0,1)$ such that
$$
T_{\beta_n}(x_n)=T_{\beta_n^n}(x_n)=x_n.
$$
Thus $\delta_{x_n}$ is a common singular invariant probability different from $\delta_0$.
\end{proposition}

\begin{proof}
Fix an integer $n\ge3$ and define
$$
S_n(t)\coloneqq\sum_{j=0}^{n-1}t^j,\qquad t\in\mathbb R.
$$
This polynomial is strictly increasing on $(0,\infty)$, and
$$
S_n(2)=2^n-1<2^n<3^{n-1}<S_n(3).
$$
Let $\beta_n\in(2,3)$ be the unique solution of $S_n( \beta_n)=2^n$. The polynomial $S_n(t)-2^n$ is monic
with integer coefficients, so a rational root would be an integer. Hence $\beta_n$ is irrational. Define
$$
x_n\coloneqq\frac1{\beta_n-1}\in(0,1).
$$
The geometric-sum identity gives
$$
\beta_n^n-1=(\beta_n-1)S_n(\beta_n)=2^n(\beta_n-1).
$$
Consequently,
$$
\beta_nx_n=x_n+1,\qquad \beta_n^nx_n=x_n+2^n.
$$
Taking fractional parts proves both fixed-point identities. Since $x_n>0$, its point mass is singular and differs from $\delta_0$.
\end{proof}

For $n=3$, the construction gives the explicit example
$$
\beta_3=\frac{\sqrt{29}-1}{2},\qquad
\beta_3^2+\beta_3=7,\qquad
\beta_3^3=8\beta_3-7.
$$
Together with Theorem~\ref{thm:square}, Proposition~\ref{prop:higher-powers} shows that, among integers $n\ge2$, every common
singular probability of $T_\beta$ and $ T_{\beta^n}$ equals $\delta_0$ for every irrational $\beta>1$ exactly when $n=2$.

Theorem~\ref{thm:square} also gives examples outside every rational affine relation. For instance, take
$$
\beta\coloneqq\sqrt[3]{2},\qquad \gamma\coloneqq\sqrt[3]{4}=\beta^2.
$$
The numbers $1,\beta,\gamma$ are linearly independent over $\mathbb Q$, since $\beta$ has degree three, but
$\common(\beta,\gamma)=\{\delta_0\}$. The same linear independence holds whenever $\beta$ is transcendental or algebraic of
degree at least three. Thus rational affine dependence is not necessary for singular rigidity.

\section{The resonant relation and further questions}\label{sec:scope}

\subsection{The excluded affine relation}

The assumptions of Theorem~\ref{thm:singular} have a symmetric formulation: at least
one base is irrational, $1,\beta,\gamma$ are linearly dependent over $\mathbb Q$, and
$$
\frac{\gamma-1}{\beta-1} \notin\mathbb Q.
$$
Indeed, after taking $\beta$ irrational, rational dependence gives a unique expression $\gamma=a\beta+c$ with rational coefficients, and
$$
\frac{\gamma-1}{\beta-1}=a+\frac{a+c-1}{\beta-1}.
$$
The ratio is irrational precisely when $a+c\ne1$. This condition differs from multiplicative independence, which concerns the logarithms of the bases.

The complementary relation $(\gamma-1)/(\beta-1)\in\mathbb Q$ is called \textbf{resonance}. In this case the coefficient $D$ in \eqref{eq:moment} is zero, so the rotation identity does not force the first moment of $\mu$ to vanish.  The counterexamples in Proposition~\ref{prop:higher-powers} satisfy
$$
\beta_n^n=2^n\beta_n-(2^n-1),\qquad
\frac{\beta_n^n-1}{\beta_n-1}=2^n.
$$
They therefore belong to this excluded affine family. Their nonzero common fixed points
show that the resonant relation cannot be added to Theorem~\ref{thm:singular} as stated.

For a quadratic irrational base, the relation $\beta^2=a\beta+c$ is always nonresonant. Otherwise $a+c=1$ would give
$$
0=\beta^2-a\beta+a-1=(\beta-1)(\beta-a+1),
$$
forcing $\beta=a-1\in\mathbb Q$. Thus Theorem~\ref{thm:square} agrees with the affine classification on their common
domain. For algebraic bases of degree at least three, and for transcendental bases, it is the positivity of the rotation identity that supplies the missing argument.

\subsection{Further questions}

The results above suggest two further questions.

First, for multiplicatively independent noninteger bases $\beta,\gamma>1$, what additional hypotheses ensure that every common
singular invariant probability is $\delta_0$? Theorem~\ref{thm:singular} gives an answer for the nonresonant rational affine family,
whereas the bases in Theorem~\ref{thm:square} are multiplicatively dependent.

Second, can one classify $\common(\beta,\gamma)$ for distinct resonant bases $\beta,\gamma>1$ with $\beta$ irrational?
Proposition~\ref{prop:higher-powers} gives examples with nonzero common fixed points, but does not determine all common invariant
probabilities in this family.

\section*{Use of generative AI}
GPT-6 (OpenAI) assisted in extending the results from adjacent bases to rational affine relations and in the proofreading, English translation and reference checking of the manuscript. The author takes full responsibility for the mathematical content and for all claims in the paper.

\end{document}